\documentclass{amsart}

\usepackage{amsmath,amsthm,amssymb}

\newtheorem{theorem}{Theorem}[section]

\newtheorem{corollary}[theorem]{Corollary}

\newcommand{\fn}[1]{\mathrm{#1}}
\newcommand{\mdl}[1]{\mathcal{#1}}

\newcommand{\ph}{\varphi}
\newcommand{\NN}{\mathbb{N}}
\newcommand{\QQ}{\mathbb{Q}}
\newcommand{\ZZ}{\mathbb{Z}}
\newcommand{\RR}{\mathbb{R}}

\newcommand{\nn}{n}

\newcommand{\st}{\; | \;}
\newcommand{\mymod}{\mathbin{\mathrm{mod}}}
\newcommand{\len}{\fn{length}}

\title{Inverting the Furstenberg correspondence}

\author{Jeremy Avigad}

\address{Departments of Philosophy and Mathematical Sciences\\
Carnegie Mellon University\\
Pittsburgh, PA 15213}

\thanks{Work partially supported by NSF grant DMS-1068829.}


\begin{document}

\begin{abstract}
Given a sequence of sets $A_n \subseteq \{0,\ldots,n-1\}$, the Furstenberg correspondence principle provides a shift-invariant measure on $2^\NN$ that encodes combinatorial information about infinitely many of the $A_n$'s. Here it is shown that this process can be inverted, so that for any such measure, ergodic or not, there are finite sets whose combinatorial properties approximate it arbitarily well. The finite approximations are obtained from the measure by an explicit construction, with an explicit upper bound on how large $n$ has to be to yield a sufficiently good approximation. 

We draw conclusions for computable measure theory, and show, in particular, that given any computable shift-invariant measure on $2^\NN$, there is a computable element of $2^\NN$ that is generic for the measure. We also consider a generalization of the correspondence principle to countable discrete amenable groups, and once again provide an effective inverse.
\end{abstract}

\maketitle

\section{Introduction}

For each $n$, let $A_n$ be a subset of $\{0, \ldots, n-1\}$. The Furstenberg correspondence principle, described more precisely in Section~\ref{preliminaries:section}, allows one to assign a shift-invariant measure on Cantor space, $2^\NN$, which encodes combinatorial information about infinitely many of the $A_n$'s. This correspondence lies at the heart of Furstenberg's remarkable ergodic-theoretic proof \cite{furstenberg:77,furstenberg:81} of Szemer\'edi's theorem, and allows one to use facts about shift-invariant measures on Cantor space to draw conclusions about subsets of $\{0,\ldots,n-1\}$, for sufficiently large $n$.

It is natural to ask exactly which shift-invariant measures on $2^\NN$ arise from this correspondence. It is not hard to show that any \emph{ergodic} measure can be obtained in this way; any generic point of the system reflects all the relevant information about the measure, and the desired finite approximations can be read off from such a point. This fact has been noted by a number of authors, including Bergelson, Furstenberg, and Weiss \cite[Section 2]{bergelson:et:al:06}, Bergelson, Leibman, and Lesigne \cite[Section 0.B]{bergelson:et:al:07}, Tao \cite[Section 2]{tao:08} and Host and Kra \cite[Proposition 6.1]{host:kra:09}.

Theorem~\ref{inverse:thm} of Section~\ref{inverting:section} shows that, in fact, \emph{every} shift-invariant measure on $2^\NN$, ergodic or not, arises in such a way. The proof provides an explicit construction of a sequence of finite combinatorial approximations to any given measure, with, moreover, a uniform upper bound on how large $n$ has to be in order to approximate the measure to a given accuracy. 

Section~\ref{computable:section} considers some consequences for computable measure theory. There is a precise sense, described in \cite{galatolo:et:al:unp:b,galatolo:et:al:unp,hoyrup:rojas:09b,weihrauch:99}, in which a dynamical system can be said to be \emph{computable}; and similarly for a transformation of such a space, and an element of the underlying space. In particular, $2^\NN$ equipped with the left shift is computable, a computable element of $2^\NN$ is a computable binary sequence, and a computable measure on $2^\NN$ is an algorithm which computes (arbitrarily good rational approximations to) the measure of each basis set in the usual topology.

It is by now well known that many common ergodic-theoretic constructions are not computable. For example, V'yugin \cite{vyugin:97,vyugin:98} has shown that one cannot generally compute a bound on the rate of convergence of a sequence of ergodic averages $A_n f$, even when $f$ is computable (see also \cite[Section 5]{avigad:et:al:10} and \cite{avigad:unp:o}). Similarly, there is a sense in which ergodic decomposition is not computable \cite{hoyrup:11}. The passage from a sequence of sets $(A_n)$ to one of the measures guaranteed to exist by the Furstenberg correspondence principle is certainly not computable, since it is not even continuous in the data.\footnote{There is, however, always such a measure that is \emph{low} in the Turing jump of the sequence $(A_n)$. In this sense the Furstenberg correspondence principle is analogous to the Bolzano-Weierstrass principle; see \cite{kreuzer:11,safarik:kohlenbach:10}.} The proof of theorem~\ref{inverse:thm} shows, however, that passage in the other direction is fully effective: one can explicitly compute a sequence of combinatorial approximations from the given measure.

This fact has a surprising consequence: every computable shift-invariant measure on $2^\NN$, ergodic or not, has a computable generic point. A number of recent papers \cite{galatolo:et:al:09,galatolo:et:al:unp:b,galatolo:et:al:unp} are concerned with identifying conditions under which a computable measure preserving system has such an element. Theorem~\ref{generic:thm} shows that not only is this always the case when the underlying dynamical system is $2^\NN$ with the shift, but, moreover, one has explicit rates of convergence that are independent of the measure in question.

The existence of generic points for shift-invariant measures on $2^\NN$ was first established by Colebrook \cite{colebrook:70}. The construction in Section~\ref{computable:section} depends on the fact that one can obtain a point of $2^\NN$ by piecing together finite specifications, and Sigmund \cite{sigmund:74} has shown that generic points exist, more generally, for measure-preserving systems satisfying the ``specification property.'' But it does not seem possible to adapt the computability results here to this more general setting; this is discussed in Section~\ref{questions:section}.

In recent years, the correspondence principle has been more broadly construed as a way of relating combinatorial configurations in a discrete group with measure-preserving systems on which this group acts. In particular, the principle has been generalized to countable discrete amenable groups in \cite{bergelson:mccutcheon:98}, \cite[Section 4]{bergelson:00}, and, more recently, \cite{bergelson:furstenberg:09}. It is noted in \cite[Section 1]{bergelson:furstenberg:09} that one can, conversely, pass in the other direction from ``space averages'' to combinatorial ``group averages'' in the case where the action of the group of the space is ergodic. Section~\ref{amenable:section} below again lifts the restriction to ergodic actions, and provides an effective version of the transformation.

I am grateful to Bryna Kra and Henry Towsner for comments and suggestions on an earlier draft; to Manfred Denker and Matthieu Hoyrup for subsequently pointing me to Sigmund's results; and, especially, to an anonymous referee for many helpful comments, corrections, suggestions, and references.

\section{Preliminaries}
\label{preliminaries:section}

For each natural number $n$, it is convenient to identify $\nn$ with the set $\{0, \ldots, n-1\}$, and to identify each subset $A$ of $\nn$ with the finite binary string of length $n$ whose $i$th digit is $1$ if and only if $i$ is in $A$. Note that this representation encodes both the set $A$ and the fact that $A$ is to be viewed as a subset of $\nn$. If $\sigma$ is another binary sequence, say that $\sigma$ \emph{occurs at position $i$ of $A$} if for every $j$ less than the length of $\sigma$, the $j$th bit of $\sigma$ agrees with the bit of $A$ at $(i + j) \mymod n$; that is, we let $\sigma$ wrap around to the beginning of $A$, if necessary, in doing the comparison. Given $A$ and $\sigma$, set
\begin{align*}
E_A(\sigma) & = \{ i < n \st \mbox{$\sigma$ occurs at position $i$ in $A$} \}, \\
N_A(\sigma) & = | E_A(\sigma) |, \quad \mbox{and}  \\
D_A(\sigma) & = N_A(\sigma) / n.
\end{align*}
So $E_A(\sigma)$ is the set of positions at which $\sigma$ occurs in $A$, 
$N_A(\sigma)$ is the number of times it occurs, and $D_A(\sigma)$ is the density of occurrences. Note that $E_A([1]) = A$.

Let $2^\NN$ denote Cantor space, that is, the space of functions from $\NN$ to the discrete space $\{0, 1\}$, under the product topology. If we view elements $\omega$ of $2^\NN$ as infinite binary sequences, it makes sense to write $\sigma \subset \omega$ to denote that the finite sequence $\sigma$ is an initial segment of $\omega$. The collection of cylinder sets $[\sigma]$ provides a basis for the topology on $2^\NN$, where $[\sigma] = \{ \omega \st \sigma \subset \omega \}$ is the set of infinite sequences extending $\sigma$. I will use $\mdl B$ to denote the Borel sets in this topology. Let $T$ denote the shift-left map on $2^\NN$, defined by setting $(T\omega)(n) = \omega(n + 1)$ for every $n$. Notice that for every $\sigma$, $T^{-1}([\sigma]) = [0\sigma] \cup [1\sigma]$, the set of infinite binary strings in which $\sigma$ occurs in the second position. Every finite subset $A$ of $\nn$ gives rise to a measure $\mu_A$ on the Borel subsets of Cantor space defined by setting $\mu_A([\sigma]) = D_A(\sigma)$ and applying the Caratheodory extension theorem. By the observations above, we have $\mu_A(T^{-1} [\sigma]) = \mu_A([\sigma])$, which is to say, $\mu$ is invariant under the shift. 

Now let $(A_n)$ be a sequence of sets of natural numbers, with $A_n \subseteq n$ for each $n$. By the compactness of the space of measures on Cantor space in the vague topology, there is a subsequence $(\mu_{A_{n_i}})_{i \in \NN}$ of $(\mu_{A_n})$ that converges weakly to a measure $\mu$ on Cantor space. In particular, for each $\sigma$, the sequence $(\mu_{A_{n_i}}([\sigma]))$, which is equal to the sequence $(D_{A_{n_i}}(\sigma))$, converges to $\mu([\sigma])$. Thus we have:
\begin{theorem}
 \label{furstenberg:correspondence:thm}
For every sequence $(A_n)$ of sets with $A_n \subseteq \nn$, there are a $T$-invariant measure $\mu$ on $2^\NN$ and a subsequence $(A_{n_i})$ of $(A_n)$ with the property that for every $\sigma$, $\mu([\sigma]) = \lim_{i \to \infty} D_{A_{n_i}}(\sigma)$.
\end{theorem}
This theorem can be proved more directly by iteratively thinning the sequence $(A_n)$ so that the densities converge for each $\sigma$, taking a diagonal subsequence, and then defining $\mu([\sigma])$ to be the resulting limit.

We can take Theorem~\ref{furstenberg:correspondence:thm} to be a precise statement of the Furstenberg correspondence principle, though sometimes the phrase is used to refer to one of its consequences. Note that since the size of $\sigma$ remains fixed as $n$ grows, the limits in question are not changed if we do not take wraparound into account when counting the number of occurrences of $\sigma$ in $A_n$.

\section{Inverting the correspondence}
\label{inverting:section}

In this section we complement Theorem~\ref{furstenberg:correspondence:thm} by showing that, in fact, any shift-invariant measure $\mu$ can be obtained as the result of the construction.

\begin{theorem}
\label{inverse:thm}
Let $\mu$ be any $T$-invariant measure on $2^\NN$. Then for each $j$ and $\varepsilon > 0$, there exist $n \leq 2^{O(j/\varepsilon)}$ and $A \subseteq \nn$ such that for every $\sigma$ of length at most $j$, $|\mu([\sigma])-D_A(\sigma)| < \varepsilon$. Moreover, there is an $m = m(j,\varepsilon)$ such that for any $n \geq m$ there is an $A \subseteq \nn$ with this property.
\end{theorem}

We can abbreviate the conclusion of the theorem by saying that $A$ gives a $(j,\varepsilon)$-good approximation to $\mu$. The first claim provides an explicit bound on how large $A$ needs to be to provide such an approximation. It would be interesting to know whether this bound can be improved. It is the second claim, however, the provides a natural inverse to Theorem~\ref{furstenberg:correspondence:thm}: if for each $j$ we choose $m_j$ large enough to ensure there are $(j,1/j)$-good approximations for any $n \geq m_j$, then for any $\mu$ we can build a sequence which contains such approximations between $m_j$ to $m_{j+1}$. Any measure satisfying the conclusion of Theorem~\ref{furstenberg:correspondence:thm} then has to coincide with $\mu$.

\begin{proof}
Fix $j$ and $\varepsilon > 0$, and let $k$ be an integer much larger than $j$ and $1 / \varepsilon$, to be specified more precisely later on. Then the set $\{ [\tau] \st \len(\tau) = k\}$ forms a partition of $2^\omega$, and if $i < k - j$, 
\[
 \mu(T^{-i}[\sigma] \cap [\tau]) = \left\{
  \begin{array}{ll}
     \mu([\tau]) & \mbox{if $\sigma$ occurs at position $i$ of $\tau$} \\
     0 & \mbox{otherwise.}
  \end{array}\right.
\]
Recall that $N_\tau(\sigma)$ denotes the number of times that $\sigma$ occurs in $\tau$. By the $T$-invariance of $\mu$, we have
\begin{align*}
 \mu([\sigma]) & = \frac{1}{k} \sum_{i < k} \mu(T^{-i}[\sigma]) \\
   & = \frac{1}{k} \sum_{i < k} \sum_\tau \mu(T^{-i}[\sigma] \cap [\tau]) \\
   & = \sum_\tau \frac{1}{k} \sum_{i < k} \mu(T^{-i}[\sigma] \cap [\tau]) \\
   & = \sum_\tau \frac{1}{k} (N_\tau(\sigma) + O(j)) \mu([\tau]) \\
   & = \sum_\tau D_\tau(\sigma) \mu([\tau]) + O(j/k).
\end{align*}
In other words, if the $\tau$'s are sufficiently long, the average of the densities of $\sigma$ in each $\tau$, weighted by $\mu([\tau])$, provide a good approximation to $\mu([\sigma])$. We now obtain the desired set $A$ by concatenating copies of the $\tau$'s in the right proportion; the fact that $k$ is much larger than $j$ will ensure that the occurrences of $\sigma$ near the border between copies of the $\tau$'s will have a negligible contribution to the overall density.

More precisely, let $l$ be much larger than $k$ and $1 / \varepsilon$, and let $\tau_0, \tau_1, \ldots, \tau_{2^k-1}$ be an enumeration of the sequences of length $k$. For each $i \leq 2^k$ let $b_i$ be the closest integer to $(\sum_{m < i} \mu([\tau_m])) \cdot l$, and let $a_{\tau_i} = b_{i+1} - b_i$. Then the values $a_\tau / l$ provide a good rational approximation to $\mu([\tau])$, with
\begin{equation*}
 \label{ai:approx:eq}
 | \mu([\tau]) - \frac{a_\tau}{l}| < \frac{2}{l} \quad \mbox{and} \quad \sum_\tau a_\tau = b_{2^k} = l.
\end{equation*}
Let $A$ be the set obtained by concatenating $a_0$ copies of $\tau_0$, followed by $a_1$ copies of $\tau_1$, and so on. Then $A$ has length $k l$. Accounting for occurrences of $\sigma$ near the border between such copies, the total number of occurrences of $\sigma$ in $A$ is given by
\[
 N_A(\sigma) = \sum_\tau a_\tau N_\tau(\sigma) + l \cdot O(j).
\]
Dividing by $kl$, we have
\begin{align*}
 D_A(\sigma) & = \sum_\tau \frac{a_\tau}{l} D_\tau(\sigma) + O(\frac{j}{k}) \\
   & = \sum_\tau \mu([\tau]) D_\tau(\sigma) + O(\frac{2^k}{l}) + O(\frac{j}{k}) \\
   & = \mu([\sigma]) + O(\frac{2^k}{l}) + O(\frac{j}{k}),
\end{align*}
using the previous expression for $\mu([\sigma])$. Now we only need to choose $k = O(j / \varepsilon)$ large enough to make the second error less than $\varepsilon / 2$, and then choose $l = O(2^k / \varepsilon) = 2^{O(j / \varepsilon)}$ to make the first error less than $\varepsilon / 2$. The length of $A$ is then $k l = 2^{O(j / \varepsilon)}$, as desired.

The last claim of the proof is easily obtained, for example, by concatenating sufficiently many copies of $A$ and truncating as necessary.
\end{proof}

As is well known, $2^\NN$ with the left shift is universal for measure-preserving systems with a distinguished set, in the following sense (see, for example, \cite[Example 2.2.6]{tao:09}). Let $\mdl X = (X, \mdl C, \nu, U)$ be a measure-preserving system, and let $E$ be a $\mdl C$-measurable set. Define a function $\ph$ from $X$ to $2^\NN$ by
\[
 (\ph(x))_i = \left\{
    \begin{array}{ll}
       1 & \mbox{if $U^i x \in E$} \\
       0 & \mbox{otherwise.}
    \end{array}
  \right.
\]
In other words, the 1's in $\ph(x)$ correspond to places where the orbit of $x$ under $U$ lands in $E$. Then for any $x \in X$, $T \ph(x) = \ph(U x)$, and $\ph^{-1}([1]) = E$. Moreover, $\ph^{-1}$, as a function on subsets of $2^\NN$, is a $\sigma$-algebra homomorphism from $\mdl B$ onto the $\sigma$-subalgebra of $\mdl C$ generated by $E$. Define the ``push-forward'' measure $\mu$ on $\mdl B$ by setting $\mu(A) = \nu(\ph^{-1}(A))$ for every $A$ in $\mdl B$. Then $\mu$ is a $T$-invariant measure, which encodes information about the measure $\nu$ on intersections of finite shifts of $E$ and their complements: if we let $(U^{-i} E)^{\sigma_i}$ denote $U^{-i} E$ if $\sigma_i = 1$ and $X \setminus U^{-i} E$ if $\sigma_i = 0$, we have
$\mu([\sigma]) = \nu(\bigcap_{i < \len(\sigma)} (U^{-i} E)^{\sigma_i})$.
This allows us to generalize the statement of Theorem~\ref{inverse:thm}:
\begin{corollary}
 \label{inverse:cor}
Let $\mdl X = (X, \mdl C, \nu, U)$ be any measure-preserving system, and let $E$ be any $\mdl C$-measurable set. Then for each $j$ and $\varepsilon > 0$, there exist $n$ and $A \subseteq \nn$ such that for every $\sigma$ of length at most $j$, 
\[
\left|\nu\left(\bigcap_{i < \len(\sigma)} (U^{-i} E)^{\sigma_i}\right) - D_A(\sigma)\right| < \varepsilon.
\]
\end{corollary}

\section{Consequences for computable measure theory}
\label{computable:section}

We can also consider Theorem~\ref{inverse:thm} in computability-theoretic terms. This presupposes some notions from computable analysis and measure theory; I will sketch the necessary background here, and refer the reader to \cite{brattka:et:al:08,hoyrup:rojas:09b,weihrauch:99} for details. 

Computability theory starts with the notion of a \emph{computable} function from the natural numbers to natural numbers, or from finite strings of symbols to finite strings of symbols. One then obtains notions of computability with respect to other finitary objects (integers, pairs of numbers, finite graphs, and so on) by fixing encodings of these objects as numbers or strings. Intuitively, a function from a set of finitary objects to another is said to be computable if there is an algorithm, or computer program, that computes it. This notion can be made precise using the Turing machine model of computation and fixing the various encodings, but for practical purposes the intuitive description suffices.

Computable analysis has to take into account the representation of infinitary objects, like the real numbers, which cannot be encoded with a finite amount of data. We define a real number $r$ to be computable if there is a computable function $f$ from $\NN$ to $\QQ$ such that for every $i$, $|r - f(i)| < 2^{-i}$. In other words, $r$ is computable if one can compute arbitrarily good rational approximations to it. Notice that the choice of $i \mapsto 2^{-i}$ as a rate of convergence is somewhat arbitrary, and the definition is unchanged if one replaces $2^{-i}$ with any computable sequence of rationals that decreases to $0$; from any such representation we can obtain any other. Notice that we have defined a real number to be computable if it has a computable \emph{representation} as a Cauchy sequence of rationals with a fixed rate of convergence; a given computable real will have multiple computable representations.

How shall we define a computable function from $\RR$ to $\RR$? The problem is that the inputs to such a function are no longer finite objects. The standard solution is to say that a function $F$ from $\RR$ to $\RR$ is computable if there is an algorithm which, on input $i$, is allowed to ask for rational approximations of the input $x$ to any desired accuracy and, after finitely many such queries, terminates and returns an approximation of $F(x)$ to within $2^{-i}$. The notion can be made precise in terms of a Turing machine with access to an oracle tape that contains a representation of the input, but, once again, for practical purposes, the intuitive description suffices. This model of computation on the real numbers was originally proposed by Grzegorczyk~\cite{grzegorczyk:57}, and is an instance of what is generally referred to as ``type 2 computability'' today \cite{brattka:et:al:08,weihrauch:99}. Notice that the algorithm computing $F$ is supposed to act appropriately on any representation of a real number $x$, whether $x$ is computable or not. One can show that any computable function $F$ from $\RR$ to $\RR$ is continuous; roughly speaking, this holds because finite approximations to the value of $F(x)$ depend on only a finite amount of the data representing $x$. 

There is nothing special about the real numbers; the method carries over to any system of elements that can be given the structure of a separable metric space. Suppose $(X, d)$ is such a space and $A$ is a countable dense subset of $X$. Assuming one has finitary representations of the elements of $A$ such that the distances between these elements are computable, then if one replaces $\RR$ and $\QQ$ in the preceding discussion by $X$ and $A$, respectively, one obtains notions of a ``computable element of $X$'' and a ``computable function on $X$.'' We will not need the full generality of these definitions here. Instead, I will focus on how they play out for $2^\NN$ and the space of measures on $2^\NN$.

An element $\omega$ of $2^\NN$ is computable if and only if the function from $\NN$ to $\{0, 1\}$ which, on input $i$, returns the $i$th digit of $\omega$ is computable. Similarly, a computable function $T$ from $2^\NN$ to $2^\NN$ is given by an algorithm which, for every $i$, computes the $i$th bit of $T \omega$ after querying finitely many bits of $\omega$. For example, if $T$ is the left shift, then $T$ is easily seen to be computable, because in order to output the $i$th bit of $T\omega$ one need only query the $(i+1)$st bit of $\omega$.

A measure $\mu$ on $2^\NN$ is said to be computable if there is an algorithm which, on input $\sigma$, computes (arbitrarily good rational approximations to) $\mu([\sigma])$. In other words, $\mu$ is computable if there is an algorithm which, on input $\sigma$ and $i$, computes a rational approximation of $\mu([\sigma])$ to within $2^{-i}$. More generally, one can take an arbitrary measure $\mu$ to be represented by such a function from $S \times \NN$ to $\QQ$, where $S$ is the set of finite binary strings. As in the case of the real numbers, it makes sense to talk about algorithms that carry out computations relative to such a representation. 

The following theorem provides a sense in which, from a computational point of view, a measure on $2^\NN$ is ``morally equivalent'' to a sequence of $(j,\varepsilon)$ good approximations.

\begin{theorem}
  There are a computable function $m(j,\varepsilon)$ and an algorithm with the following property: given any representation of a measure $\mu$ on $2^\NN$, the algorithm computes a sequence $(A_n)$ of subsets of $n$ such that for every $n \geq m(j,\varepsilon)$, $(A_n)$ is an $(j,\varepsilon)$-good approximation to $\mu$. Conversely, there is an algorithm which, given a representation of such a function $m$ and sequence $(A_n)$, computes the measure $\mu$.
\end{theorem}

\begin{proof}
 Let $m(j, \varepsilon)$ be as in the statement of Theorem~\ref{inverse:thm}. The proof of Theorem~\ref{inverse:thm} provides the requisite algorithm, that is, for each $n$, an explicit description of how to obtain $A \subseteq n$ from finitely many values of $\mu$ on cylinder sets. Conversely, given $m$ and $(A_n)$, to compute $\mu([\sigma])$ to within $\varepsilon$, let $j = \len(\sigma)$, choose $n = m(j, \varepsilon)$, and compute $D_{A_n}(\sigma)$.
\end{proof}

If $\mu$ is a shift-invariant measure on $2^\NN$, a point $\omega$ of $2^\NN$ is \emph{generic} if for every finite binary sequence $\sigma$, 
\[
 \mu([\sigma]) = \lim_{n \to \infty} \frac{1}{n} \sum_{i < n} 1_{[\sigma]} (T^i \omega).
\]
In other words, for every $\sigma$, the limiting frequency of occurrences of $\sigma$ in $\omega$ is $\mu([\sigma])$ (see, for example, \cite{hoyrup:11}). The following theorem shows that given a shift-invariant measure $\mu$ on $2^\NN$, one can compute a generic point, such that the rate of convergence of the limit above is moreover computable (and independent of $\mu$).

\begin{theorem}
\label{generic:thm}
 There is a computable function $m(j,\varepsilon)$ with the following property. Given a representation of a shift invariant measure $\mu$ on $2^\NN$, one can compute a point $\omega$ that is generic for $\mu$, with the additional property that for every $\sigma$ of length $j$, every $\varepsilon > 0$, and every $n \geq m(j,\varepsilon)$, $| \mu([\sigma]) - \frac{1}{n} \sum_{i < n} 1_{[\sigma]} (T^i \omega) | < \varepsilon$.
\end{theorem}

\begin{proof}
 Given $\mu$, for each $j$ let $A_j$ provide a $(j,2^{-j})$-good approximation to $\mu$ with length bounded as in Theorem~\ref{inverse:thm}. The idea is to build $\omega$ by concatenating copies of $A_1$, then copies of $A_2$, then copies of $A_3$, and so on, choosing enough copies at each stage to ensure that the transitions are smooth.
 Specifically, construct $\omega = \tau_0 \tau_1 \tau_2 \ldots$ in stages, as follows. First, define $\tau_0$ to be the empty sequence. Now, assuming $\tau_0, \ldots, \tau_l$ are defined, set $m_l = \len(\tau_0 \tau_1 \ldots \tau_l)$, and let $k = m + \len(A_{l+2})$. Let $\tau_{l+1}$ be the concatenation of enough copies of $A_{l+1}$ so that $k / \len(\tau_{l + 1}) < 2^{-(l+1)}$. Then a routine calculation shows that for every $\sigma$ of length at most $l + 1$ and $n \geq m_{l+1}$, $|\mu([\sigma]) - \frac{1}{n} \sum_{i < n} 1_{[\sigma]} (T^i \omega)| < 2^{-l}$. Clearly $\omega$ can be computed from $\mu$, and a bound $m(j,\varepsilon)$ on $m_{\max(j,\lceil \log_2(\varepsilon^{-1})\rceil + 1)}$ can be computed outright, independent of $\mu$.
\end{proof}

\section{An extension to amenable groups}
\label{amenable:section}

In recent years the correspondence principle has typically been construed more abstractly as a way of relating combinatorial configurations in a discrete group with measure-preserving systems on which this group acts. The principle has been generalized to countable discrete amenable groups in \cite{bergelson:mccutcheon:98}, \cite[Section 4]{bergelson:00}, and even more broadly in \cite{bergelson:furstenberg:09}. (See also \cite{towsner:unp}.) The conventional way of passing in the other direction, from ``space averages'' to ``group averages,'' relies on the pointwise ergodic theorem and works only for ergodic measures. In this section, we provide an effective proof that once again avoids the assumption of ergodicity.

\newcommand{\symdiff}{\mathop{\Delta}}

A countable discrete group $\Gamma$ is said to be \emph{amenable} if for every finite $K \subset \Gamma$ and $\varepsilon > 0$ there is a finite $F \subset \Gamma$ such that $| F \symdiff K F | < \varepsilon \cdot | F |$. Given such a $\Gamma$, we can fix a sequence $F_0 \subseteq F_1 \subseteq F_2 \subseteq \ldots \subseteq \Gamma$ such that $\bigcup_i F_i = \Gamma$ and for every finite set $K$ there is an $i$ such that $|F_j \symdiff K F_j| < \varepsilon \cdot |F_j|$ for every $j \geq i$. Such a sequence is called a \emph{F{\o}lner sequence}. 

Here the natural analogue to $2^\NN$ is $2^\Gamma$ under the product topology. For each $\gamma \in \Gamma$, $\gamma$ gives rise to the action $T_\gamma$ on $2^\Gamma$ defined by $(T_\gamma \omega)(\alpha) = \omega(\gamma \alpha)$. A measure $\mu$ on $2^\Gamma$ is said to be $\Gamma$-invariant if $T_\gamma$ preserves $\mu$ for each $\gamma$. On the natural numbers, $(\{0,\ldots,n-1\})_n$ forms a F{\o}lner sequence, and it is natural to associate each element of that sequence with the corresponding cyclic subgroup. In general, however, there is no way to associate a group to each element $F_n$ of a F{\o}lner sequence, nor a way to paste copies of such groups together. As a result, we need a more general framework.

Fix $\Gamma$. Given a finite subset $F$ of $\Gamma$ and a set $X$, we define a \emph{partial action of $F$ on $X$} to consist of a partial function $x \mapsto \gamma x$ for each $\gamma$ in $F$, satisfying $1 x = x$ for every $x$, and $\gamma (\gamma' x) = (\gamma \gamma') x$ whenever $\gamma$, $\gamma'$, and $\gamma \gamma'$ are all in $F$, and both sides of the equation are defined. Say that the \emph{domain} of $F$ with respect to this partial action is the intersection of the domains of the $\gamma$'s, as $\gamma$ ranges over the elements of $F$. In other words, an element $i \in X$ is in the domain of the partial action if $\gamma i$ is defined for each $i$ in $F$.

A pattern, $\sigma$, is now a map from some finite subset of $\Gamma$ to $\{0,1\}$. As above, the standard topology on $2^{\Gamma}$ is generated by the cylinder sets $[\sigma]$, where $[\sigma] = \{ \omega \st \mbox{$\omega(i) = \sigma(i)$ for every $i \in \fn{dom} (\sigma)$} \}$. Fix a finite subset $F$ of $\Gamma$ and an action of $F$ on some finite set $X$. If $A$ is a subset of $X$ and $i$ is an element of $X$, say that \emph{$\sigma$ occurs at position $i$ in $A$} if and only if for every $\alpha$ in the domain of $\sigma$, $\sigma(\alpha) = 1$ if and only if $\alpha i$ is defined and in $A$. As in Section~\ref{preliminaries:section}, define
\begin{align*}
N_A(\sigma) & = | \{ i \in X \st \mbox{$\sigma$ occurs at position $i$ in $A$} \} |, \quad \mbox{and}  \\
D_A(\sigma) & = N_A(\sigma) / n,
\end{align*}
where the set $X$ and the partial action on $X$ are left implicit in the notation. 

The following theorem provides one formulation of the correspondence principle for amenable groups. 

\begin{theorem}
 Let $\Gamma$ be a countable discrete amenable group, with F{\o}lner sequence $(F_n)$. Let $(X_n)$ be a sequence of sets, where each $X_n$ equipped with a partial action of $F_n$ such that $\lim_n |\fn{dom}(F_n)| / | X_n | = 1$. Then for any sequence of sets $(A_n)$, where $A_n \subseteq X_n$ for each $n$, there are a $\Gamma$-invariant measure $\mu$ on $2^\Gamma$ and a subsequence $(A_{n_i})$ of $(A_n)$ with the property that for every pattern $\sigma$, $\mu([\sigma]) = \lim_{i \to \infty} D_{A_{n_i}}(\sigma)$.
\end{theorem}

Taking $\Gamma = \ZZ$ and $F_n = \{ -(n-1), \ldots, n-1 \}$ for each $n$ yields a version of Theorem~\ref{furstenberg:correspondence:thm} with $\ZZ$ in place of $\NN$. In the formulation in \cite[Section 4]{bergelson:00}, for example, the sets $X_n$ are taken to be subsets of $\Gamma$ itself.

\begin{proof}
  As in the proof of Theorem~\ref{furstenberg:correspondence:thm}, we can iteratively thin the sequence $(A_n)$ and diagonalize so that the limit in question exists for each $\sigma$, and then define $\mu([\sigma])$ accordingly. We only need to show that $\mu$ is additive and $\Gamma$ invariant, at which point we can apply the Caratheodory extension theorem.

  To see that $\mu$ is additive, let $\sigma$ be any pattern, $\alpha$ an element of $\Gamma$ that is not in the domain of $\sigma$, and let $\sigma_0$ and $\sigma_1$ be the patterns extending $\sigma$ with value $0$ and $1$, respectively, at $\gamma$. It suffices to show that $\mu([\sigma]) = \mu([\sigma_0]) + \mu([\sigma_1])$. But since $(F_n)$ is a F{\o}lner sequence, $\fn{dom}(\sigma) \cup \{ \gamma \} \subseteq F_n$ for sufficiently large $n$, and the desired conclusion follows from the fact that $\lim_n |\fn{dom}(F_n)| / | X_n | = 1$. Similarly, $\Gamma$-invariance also follows from the fact that $(F_n)$ is a F{\o}lner sequence with this last property.
\end{proof}

We have the following inverse:

\begin{theorem}
\label{amenable:inverse:thm}
Let $\Gamma$ be a countable discrete amenable group with F{\o}lner sequence $(F_n)$, and let $\mu$ be any $\Gamma$-invariant measure on $2^\Gamma$. Then for each $j$ and $\varepsilon > 0$, there exist an $n$, a finite set $X$, a partial action of $F_n$ on $X$, and $A \subseteq X$ such that for every $\sigma$ with domain $F_j$, $|\mu([\sigma])-D_A(\sigma)| < \varepsilon$.
\end{theorem}

\begin{proof}
 The proof is similar to that of Theorem~\ref{inverse:thm}, \emph{mutatis mutandis}. Given $j$, we can choose $k$ large enough to make $|F_j| \cdot | F_j F_k \symdiff F_k | / | F_k |$ arbitrarily small. Then for any $\sigma$ with domain $F_j$, if we let $\tau$ range over patterns with domain $F_k$, we have
\begin{align*}
 \mu([\sigma]) & = \frac{1}{|F_k|} \sum_{i \in F_k} \mu(T_i^{-1}[\sigma]) \\
   & = \frac{1}{|F_k|} \sum_{i \in F_k} \sum_\tau \mu(T_i^{-1}[\sigma] \cap [\tau]) \\
   & = \sum_\tau \frac{1}{|F_k|} \sum_{i \in F_k} \mu(T_i^{-1}[\sigma] \cap [\tau])
\end{align*}
Let $F_j$ as act partially on $F_k$ by left multiplication, and view $\tau$ as representing a subset of $F_k$. Then as long as $i$ is in the domain of $F_j$, $T_i^{-1}[\sigma] \cap [\tau]$ is equal either to $[\tau]$ or the empty set, depending on whether $\sigma$ occurs at position $i$ in $\tau$. But $i$ fails to be in the domain of $F_j$ only when $\gamma i \not\in F_k$ for some $\gamma$ in $F_j$, and so the set of $i$ that are not in the domain of $F_j$ has cardinality at most $|F_j| \cdot | F_j F_k \symdiff F_k |$. Thus we can continue the calculation above,
\begin{align*}
\ldots & = \sum_\tau \frac{1}{|F_k|} (N_\sigma(\tau) + O(|F_j||F_j F_k \symdiff F_k|)) \mu([\tau]) \\
   & = \sum_\tau D_\sigma(\tau) \mu([\tau]) + O(|F_j|\cdot |F_j F_k \symdiff F_k|)/|F_k|).
\end{align*}
Now proceed as in the proof of Theorem~\ref{inverse:thm}. Let $X$ be a disjoint union of copies of $F_k$, where $F_j$ acts on $F_k$ by left multiplication, insofar as the results of the multiplication land in $F_k$. Let $A$ be a disjoint union of copies of the various $\tau$'s living on the various $F_k$'s, where the fraction of occurrences of a given $\tau$ approximates $\mu([\tau])$, and $|F_j|\cdot |F_j F_k \symdiff F_k|)/|F_k|$ is sufficiently small to preserve the quality of the approximation.
\end{proof}

As in the proof of Corollary~\ref{inverse:cor}, the universality of $2^\Gamma$ means that the result can be pulled back to arbitrary $\Gamma$-invariant spaces.

\begin{corollary}
 \label{amenable:inverse:cor}
Let $\Gamma$ be a countable discrete amenable group with F{\o}lner sequence $(F_n)$. Let $\mdl X = (X, \mdl C, \nu, \Gamma)$ be a measure-preserving system, and let $E$ be any $\mdl C$-measurable set. Then for each $j$ and $\varepsilon > 0$, there exists a finite set $X$, a partial action of $F_j$ on $X$, and $A \subseteq X$ such that for every pattern $\sigma$ with domain $F_j$, 
\[
\left|\nu\left(\bigcap_{i \in F_j} (T_i^{-1}E)^{\sigma_i}\right) - D_A(\sigma)\right| < \varepsilon.
\]
\end{corollary}

\section{Final comments}
\label{questions:section}

The proof of Theorem~\ref{inverse:thm} relies on the fact that one can construct a point of $2^\NN$ by concatenating finite specifications of its orbit behavior. This is an instance of a more general property that some dynamical systems enjoy, known as the \emph{specification property} \cite{sigmund:74}. (A slightly stronger version is presented in \cite{katok:hasselblatt:96}.) Let $(X, T)$ be a dynamical system, where $X$ is a compact space and $T$ is a continuous map from $X$ to itself. Let $\mu$ be a $T$-invariant probability measure defined on the Borel subsets of $X$. A point $x$ in $X$ is said to be generic for $\mu$ if for every continuous function $f$ from $X$ to $\RR$, $\lim_{n \to \infty} \frac{1}{n} \sum_{i < n} f(T^i x) = \int f \; d\mu$. The results of Sigmund \cite{sigmund:74} show that if $(X, T)$ satisfies the specification property, then are generic points for any $T$-invariant measure, whether it is ergodic or not. The existence of generic points for any shift-invariant measure on $2^\NN$ is a special case of this result. 

The notions of computability discussed in Section~\ref{computable:section} can be extended to more general compact metric spaces; see, for example, \cite{galatolo:et:al:unp:b,galatolo:et:al:unp,hoyrup:rojas:09b,weihrauch:99}. In analogy to Theorem~\ref{generic:thm}, one might expect that systems satisfying a computable version of the specification property will always have computable generic points. But the methods of Sigmund~\cite{sigmund:74} do not seem to translate to the computable setting: the analogue to Theorem~\ref{inverse:thm} above is given by Lemma 1 of Sigmund~\cite{sigmund:70}, which relies on the pointwise ergodic theorem in an essential way. In contrast, the proof of Theorem~\ref{inverse:thm} relies on particular features of $2^\NN$. It seems likely that Sigmund's result is noneffective, which is to say, there are computable dynamical systems satisfying a computable version of the specification property, but lacking any computable generic points. It would therefore be interesting to know the extent to which the methods used here can be generalized.

\end{document}